\documentclass[12pt,reqno]{amsart}
\usepackage{fullpage}
\usepackage{amssymb}
\usepackage{amsthm}
\usepackage{amsmath}
\usepackage[usenames]{color}
\usepackage{graphicx}
\usepackage[colorlinks=true,
linkcolor=webgreen,
filecolor=webbrown,
citecolor=webgreen]{hyperref}
\usepackage{hyperref}
\definecolor{webgreen}{rgb}{0,.5,0}
\definecolor{webbrown}{rgb}{.6,0,0}
\definecolor{red}{rgb}{1,0,0}

\usepackage{color}
\newtheorem{theorem}{Theorem}[section]

\newtheorem{conjecture}[theorem]{Conjecture}
\newtheorem{corollary}[theorem] {Corollary}
\newtheorem{definition}[theorem]{Definition}

\newtheorem{lemma} [theorem]{Lemma}

\newtheorem{question}[theorem]{Question}

\newtheorem{proposition}[theorem]{Proposition}
\newtheorem{remark}[theorem]{Remark}

\begin{document}

\title{On Diophantine tuples involving star numbers}
\author[Tapas Chatterjee]{Tapas Chatterjee}

\address{%
Department of Mathematics,\\Indian Institute of Technology Ropar, Punjab, India.}

\email{tapasc@iitrpr.ac.in}

\author{Sagar Mandal$^{*}$}
\address{Department of Mathematics\\Indian Institute of Technology Ropar, Punjab, India.}
\email{sagar.25maz0008@iitrpr.ac.in, sagarmandal31415@gmail.com}
\thanks{$^{*}$The second author is financially supported by the University Grants Commission (UGC), Government of India, through the award of the Junior Research Fellowship (JRF) with ref. No. 241620111598.}
\maketitle
\let\thefootnote\relax

\maketitle
\let\thefootnote\relax

\begin{abstract}
In this paper, we discuss two related Diophantine tuple problems involving star numbers $S_n=6n(n-1)+1$. First, we investigate classical $D(k)$-tuples whose elements are star numbers. We construct infinitely many $D(k)$-pairs for several positive and negative admissible values of $k$, and we prove several nonexistence results for $D(k)$-triples and quadruples. Second, we introduce a modified notion of a star number $D(k)$-tuple, requiring $a_ia_j+k$ itself to be a star number. We construct several infinite families of $D(1)$-triples, and study their extensions to quadruples. In particular, we prove that $\{4,108,165\}$ has the unique extension $\{4,108,165,360\}$ to a star number $D(1)$-quadruple. The main tools used in this paper are generalized Pell equations, Nagell's theorem, lower bounds for linear forms in logarithms due to Matveev and Dujella-Peth{\H{o}} version of the Baker–Davenport reduction method in diophantine approximation.
\end{abstract}

\section{Introduction}
The theory of the Diophantine tuple started with the observation of Diophantus, who noted that the product of any two distinct elements of the set $\{\dfrac{1}{16},\dfrac{33}{16},\dfrac{17}{4},\dfrac{105}{16}\}$ increased by $1$ is a square of a rational number. Fermat later found the classical Diophantine quadruple $\{1,3,8,120\}$ and Euler subsequently showed that there exist infinitely many  Diophantine quadruple. Later, several generalizations of this problem were observed, so we have the following definition
\begin{definition}
Let $k$ be a nonzero integer. A set of $m$ distinct positive integers is called a $D(k)$ m-tuple if the product of any two of its distinct elements increased by $k$ is a perfect square.    
\end{definition}
When $m=2,3,4$, such sets are referred to as Diophantine $D(k)$-pairs, triples and
quadruples, respectively. One of the fundamental questions in this area concerns the possible size of such sets. Dujella \cite{conjecture} proved that if an integer $k\not\equiv 2 \pmod{4}$ and $k\not\in\{-4,-3,-1,3,5,8,12,20\}$, then there exists at least one $D(k)$-quadruple. Further, he made the following conjecture:
\begin{conjecture}[Dujella]\label{con 1}
If $k\in\{-4,-3,-1,3,5,8,12,20\},$ then there does not exist a $D(k)$-quadruple.    
\end{conjecture}
He, Togb\'e and  Ziegler \cite{r1} proved the folklore conjecture of non-existence of quintuple when $k=1$. On the other hand, Bonciocat, Cipu and Mignotte \cite{BCM} proved the conjecture of non-existence quadruple when $k=-1,-4$, whereas Trebje\v{s}anin and Filipin \cite{D4} proved that there do not exist $D(4)$-quintuple. The Fibonacci sequence $\{F_n:n\geq0\}$ (\href{https://oeis.org/A000045}{A000045}) is defined by $F_0=0,F_1=1,$ and $ F_{n+2}=F_{n+1}+F_n$ for $n\geq0$. The Pell sequence $\{P_n:n\geq0\}$ (\href{https://oeis.org/A000129}{A000129}) is defined by 
$$P_n = 
\begin{cases}
0 & \text{if } n = 0, \\
1 & \text{if } n = 1, \\
2P_{n-1} + P_{n-2} & \text{if } n\geq2.\end{cases}$$
Diophantine triples and quadruples involving Fibonacci and Pell numbers have been studied in \cite{Fib1,Fib2,pell1,pell2}. In particular, the non-existence of Diophantine quadruples consisting of Fibonacci or Pell numbers was established in \cite{Fib1,pell2}. A Triangular number is a number of the form $T_n=\dfrac{n(n+1)}{2},$ where $n\geq1$ is an integer. One can see that \href{https://oeis.org/A000217}{A000217} enumerates triangular number. The study of Diophantine $m$-tuples involving triangular numbers began in \cite{1st}. Subsequent work by Bliznac Trebje\v{s}anin \cite{r4,r5} addressed the properties $D(u^2)$ and $D(a)$, respectively. Further developments can be found in \cite{r3,r6, Hamtat}. More than $600$ papers have been published on this topic.  The latest progress and all references can be found at Dujella's webpage \cite{homepage}.

Motivated by these developments involving Fibonacci, Pell, triangular numbers, in this paper, we consider the Diophantine tuple problem in the setting of star numbers. The $n$-th star number $S_n$ is defined as $S_n=6n(n-1)+1$, where $n\geq1$ is an integer. It is easy to observe that $S_n=1+12T_{n-1}$. One can see that \href{https://oeis.org/A003154}{A003154} enumerates star numbers. 

The following proposition establishes the form of admissible integers $k\neq0$ for which $\{S_a,S_b\}$ forms a $D(k)$-pair.
\begin{proposition}\label{Theorem 1}
If $\{S_a,S_b\}$ is a $D(k)$-pair, then $k\equiv0,3,8,11\pmod{12}.$
\end{proposition}
\begin{proof}
Since $\{S_a,S_b\}$ is a $D(k)$-pair, we have some integer $m$ such that $S_aS_b+k=m^2$. Since $m^2\equiv 0, 1, 4, 9\pmod{12}$, and $S_a, S_b\equiv 1 \pmod{12}$, we have $k\equiv 0, 3, 8, 11\pmod{12}$.
\end{proof}
In this paper, we mainly consider certain admissible values of $k$ that is $k\in \{-12,-9,-4,\\-1,3,8,11,12\}$. Let $C_{\pm}=S_a+S_b\pm2\sqrt{S_aS_b+k}$ be an integer such that $\{S_a,S_b,C_{\pm}\}$ is $D(k)$-triple, if such extension exists for the set $\{S_a,S_b\}$ such that $S_{c_{\pm}}=C_{\pm}$, we call it star regular extension. But any star pair has no star regular extension because if such an extension exists, then
$$S_aS_b+k\equiv k+1 \pmod{6},$$
which gives
$$\sqrt{S_aS_b+k}\pmod{6}=\begin{cases}
1,5 & k=-12,12 \\
2,4 & k=-9,3, \\
3 & k=-4,8, \\
0& k=-1,11.
\end{cases}$$
Therefore,
$$1=S_{c_+}\pmod{6}=\begin{cases}
4,0 & k=-12,12 \\
0,4 & k=-9,3, \\
2 & k=-4,8, \\
2& k=-1,11,
\end{cases}$$
and
$$1=S_{c_-}\pmod{6}=\begin{cases}
0,4 & k=-12,12 \\
4,0 & k=-9,3, \\
2 & k=-4,8, \\
2& k=-1,11,
\end{cases}$$
which is a contradiction.

Our first theorem in this direction establishes the existence and nonexistence of pairs $\{S_a,S_b\}$.
\begin{theorem}\label{Theorem 2}
 For $k=3,8,12,-1,-4,-9,-12$, there exist infinitely many $D(k)$-pairs containing star numbers, and  there exists no $D(11)$-pair containing star numbers.   
\end{theorem}
The next theorem produces necessary lower bounds for the index $b$ in terms of $a$ for $D(k)$-pairs.
\begin{theorem}\label{Theorem 3}
   Let $2\leq a<b$ be two positive integers and suppose that $\{S_a,S_b\}$ is a $D(k)$-pair. Then
  \begin{enumerate}
      \item $b=a+1$ or $b\geq 4a+2$ for $k=3$,
      \item $b\geq 4a+2$ for $k=8,-1,-9,-12$,
       \item $b=a+2$ or $b\geq 4a+2$ for $k=12$,
    \item $a=2,b=9$ or $b\geq 4a+2$ for $k=-4$.
  \end{enumerate}
  \end{theorem}
In the next theorem, we use the above theorem to get infinitely many $D(k)$-pairs of star numbers excluding $S_1$, where $k=3,8,12,-1,-4,-9,-12$.

\begin{theorem}\label{Theorem 4}
 For $k=3,8,12,-1,-4,-9,-12$, there exist infinitely many $D(k)$-pairs of star numbers excluding $S_1$.  
\end{theorem}
Our next result establishes the non-existence of Diophantine $D(k)$-triples containing star numbers for certain admissible integers $k\neq0$.
\begin{theorem}\label{Theorem 5}
If $k\equiv 12,20 \pmod{24}$ or $k\equiv 23,32 \pmod{36}$, then $\{S_a,S_b,S_c\}$ cannot be a $D(k)$-triple.
\end{theorem}
As an immediate consequence, we have the following corollary.
\begin{corollary}\label{Corollary 1}
For $k=-12,-4,12$, there are no integers $a,b,c$ such that the set $\{S_a,S_b,S_c\}$ is a $D(k)$-triple.    
\end{corollary}
\begin{remark}
For $k=12$, Corollary \ref{Corollary 1} supports Conjecture \ref{con 1}.
\end{remark}
In view of the computation evidence, we strongly believe that the following conjecture is true: 
\begin{conjecture}\label{con 2}
For $k=3,8,-9,-1$, there exists no $D(k)$-triple of the form $\{S_a,S_b,S_c\}$.   
\end{conjecture}
\begin{proposition}
For $k=3$, Conjecture \ref{con 1} implies Conjecture \ref{con 2}.    
\end{proposition}
\begin{proof}
 If possible, assume that Conjecture \ref{con 1} holds, but Conjecture \ref{con 2} is false for $k=3$. Then, there exist integers $a,b,c\geq1$ such that the set $\{S_a,S_b,S_c\}$ is a $D(3)$-triple. Since $6$ is not a star number, it is easy to check that $\{6,S_a,S_b,S_c\}$ is a $D(3)$-quadruple, but this contradicts our assumption on Conjecture \ref{con 1}. This completes the proof.  
\end{proof}
However, there are several values of $k$ for which $\{S_a,S_b,S_c\}$ is a $D(k)$-triple, for example, for $(a,b,c,k)\in \{(1, 2, 57, 1011),(1, 3, 126, 363),(1 ,6 ,65 ,1935), (1 ,10, 362 ,903)\}$.

From Theorem \ref{Theorem 3}, it is easy to prove that for $a\geq2$, $\{S_a,S_{a+1}\}$, $\{S_a,S_{a+2}\}$ are $D(3)$-pair and $D(12)$-pair, respectively (see proof of Theorem \ref{Theorem 4}). Corollary \ref{Corollary 1} proves that $\{S_a,S_{a+2}\}$ cannot be extended to $D(12)$-triple. The following theorem verifies a special case of Conjecture \ref{con 2}.
\begin{theorem}\label{Theorem 6}
Let $a>b\geq1$ are integers. Then $\{S_b,S_a,S_{a+1}\}$ cannot be a $D(3)$-triple. In particular, $\{S_{a-1},S_a,S_{a+1}\}$ cannot be a $D(3)$-triple.
\end{theorem}

On the other hand, an analogous problem involving classical Diophantine tuples has also been studied extensively in the setting of triangular numbers. The definition of a triangular Diophantine $D(k)$-tuple is given below.
\begin{definition}
Let $k$ be a nonzero integer. A set of $m$ distinct positive integers is called a triangular Diophantine $D(k)$ m-tuple if the product of any two of its distinct elements increased by $k$ is a triangular number.
\end{definition}
In this set up, Filipin and Szalay \cite{tri1} studied triangular Diophantine tuples containing $\{1,2\}$. They proved that no triangular $D(1)$-quadruple of the form $\{1,2,a,b\}$ exists. Filipin, Jurasi\'c, and Szalay \cite{tri2} studied the extension of the triangular $D(3)$-pair $\{3,6\}$ and proved that it cannot be extended to a quadruple of the form $\{3,6,a,b\}$. Later, Filipin, Jurasi\'c, and Togb\'e \cite{tri4} studied the triangular $D(k)$-pairs $\{k,2k\}$ for $k=10$ and $k=28$. They proved that neither pair can be extended to a triangular quadruple of the form $\{k,2k,a,b\}$ for these values of $k$. Further developments can be found in \cite{tri3,tri5}.

Motivated by these developments in the study of triangular Diophantine $D(k)$-tuples, we consider the following analogous problem involving star numbers.
\begin{definition}
Let $k$ be a nonzero integer. A set of $m$ distinct positive integers is called a star Diophantine $D(k)$ m-tuple if the product of any two of its distinct elements increased by $k$ is a star number. 
\end{definition}
In this direction, our main theorems are stated below.
\begin{theorem}\label{Theorem a+1}
For all
$$a\in\bigg\{ \frac{5\xi_{1,k-1}+12\xi_{2,k-1}-1}{2}:k\geq 1,~~\xi_{1,0}=17,\xi_{2,0}=7\bigg\},$$
where
$$\xi_{1,k}=5\xi_{1,k-1}+12\xi_{2,k-1},\quad \xi_{2,k}=2\xi_{1,k-1}+5\xi_{2,k-1},$$ 
the set
$\{a,a+1,12(2a+1)\}$  is a star Diophantine $D(1)$-triple. Further, if $12\nmid b$, then $\{a,a+1,12(2a+1),b\}$ cannot be a star Diophantine $D(1)$-quadruple.
\end{theorem}

\begin{theorem}\label{Theorem a+2} For all
$$a\in \bigg\{\frac{5\xi_{1,k-1}+12\xi_{2,k-1}-2}{2}:k\geq1,~~\xi_{1,0}=22,\xi_{2,0}=9\bigg\},$$
where 
$$\xi_{1,k}=5\xi_{1,k-1}+12\xi_{2,k-1},\quad \xi_{2,k}=2\xi_{1,k-1}+5\xi_{2,k-1},$$
the set $\{a,a+2,6(a+1)\}$  is a star Diophantine $D(1)$-triple. Further, if $6\nmid b$, then $\{a,a+2,6(a+1),b\}$ cannot be a star Diophantine $D(1)$-quadruple.
\end{theorem}

\begin{theorem}\label{Theorem a+3} For all
$$a\in \bigg\{\frac{5\xi_{1,k-1}+12\xi_{2,k-1}-3}{2}:k\geq1,~~\xi_{1,0}=27,\xi_{2,0}=11\bigg\},$$
where 
$$\xi_{1,k}=5\xi_{1,k-1}+12\xi_{2,k-1},\quad \xi_{2,k}=2\xi_{1,k-1}+5\xi_{2,k-1},$$
the set $\{a,a+3,\dfrac{8a+12}{3}\}$ is a star Diophantine $D(1)$-triple. Further, if $4\nmid b$, then $\{a,a+3,\dfrac{8a+12}{3},b\}$ cannot be a star Diophantine $D(1)$-quadruple.
\end{theorem}

\begin{theorem}\label{Theorem a+4} For all
$$a\in \bigg\{\frac{5\xi_{1,k-1}+12\xi_{2,k-1}-4}{2}: k\geq 1,~~\xi_{1,0}=8,\xi_{2,0}=3\bigg\},$$
where  
$$\xi_{1,k}=5\xi_{1,k-1}+12\xi_{2,k-1},\quad \xi_{2,k}=2\xi_{1,k-1}+5\xi_{2,k-1},$$
the set $\{a,a+4,\dfrac{3a+6}{2}\}$ is a star Diophantine $D(1)$-triple. Further, if $3\nmid b$, then $\{a,a+4,\dfrac{3a+6}{2},b\}$ cannot be a star Diophantine $D(1)$-quadruple.
\end{theorem}
It is easy to check that the set $\{4,108\}$ is a star Diophantine $D(1)$ pair, the following theorem describes the extendibility  of $\{4,108\}$ to a star Diophantine $D(1)$ triple.
\begin{theorem}\label{Theorem 1.7}
For all
$$a \in \bigg\{ \frac{3\left(\xi_{2,k}^{2}-1\right)}{8} : k\ge 1, \xi_{1,0}=-1, \xi_{2,0}=1 \bigg\}
\cup
\bigg\{ \frac{3\left(\xi_{2,k}^{2}-1\right)}{8} : k\ge 1, \xi_{1,0}=1, \xi_{2,0}=1 \bigg\},$$
where
$$\xi_{1,k}=26\,\xi_{1,k-1}+135\,\xi_{2,k-1},\quad
\xi_{2,k}=5\,\xi_{1,k-1}+26\,\xi_{2,k-1},$$
 the set $\{4,108,a\}$ is a star Diophantine $D(1)$-triple.
\end{theorem}
From the above theorem, it is clear that $\{4,108,165\}$ is a star Diophantine $D(1)$-triple. In the next theorem we extend it to star Diophantine $D(1)$-quadruple and we also prove that the extension is unique.
\begin{theorem}\label{Theorem 1.8}
 Let $\{4,108,165,a\}$ be a star Diophantine $D(1)$-quadruple. Then $a=360$.
\end{theorem}
As an immediate consequence, we obtain the following corollary.
\begin{corollary}\label{Theorem 1.9}
There does not exist a star Diophantine $D(1)$-quintuple of the form $\{4,108,165,a,b\}$.
\end{corollary}
In view of the computation evidence and results obtained in this article, we strongly believe that the following conjecture is true: 
\begin{conjecture}
 There exists no star Diophantine  $D(1)$-quintuple.   
\end{conjecture}
We also pose the following open question.
\begin{question}
 Does there exist a star Diophantine $D(1)$-quadruple other than $\{4,108,165,360\}$.   
\end{question} 
\section{Preliminaries}
In this section, we prove several auxiliary lemmas and recall a few useful results, which will play a pivotal role in the proof of our main theorems.
\begin{lemma}\label{Lemma 1}
Let $a,b,c,n$ be four integers with $b>a\geq1$. Then $S_aS_b+n=c^2$ is equivalent to the generalized Pell equation
\begin{align}\label{eq pell}
X^2-\left(9(2a-1)^2-3\right)Y^2=4n-3(2a-1)^2+1,    
\end{align}

where $X=2c, Y=2b-1$.
\end{lemma}
\begin{proof}
Observe that
$$S_a=6a(a-1)+1=\frac{3(2a-1)^2-1}{2},\quad S_b=6b(b-1)+1=\frac{3(2b-1)^2-1}{2}.$$
 Since $S_aS_b+n=c^2$, we have
 $$\frac{\big(3(2a-1)^2-1\big)\big(3(2b-1)^2-1\big)}{4}+n=c^2,$$
 which gives
 $$(2c)^2-(9(2a-1)^2-3)(2b-1)^2=4n-3(2a-1)^2+1.$$
 \end{proof}

\begin{lemma}\label{Lemma 2}
Under the assumptions in Lemma \ref{Lemma 1}, for fixed $a\geq1$ and $n$, if $(x,y)$ is a positive integer solution of (\ref{eq pell}), then $S_aS_b+n=c^2$ is satisfied for each
$$b \in \left\{ \frac{\xi_{2,k}+1}{2} : k\ge 0 \right\},$$
where  $\xi_{1,0}=x,\xi_{2,0}=y$, and for $k\geq1$
$$\xi_{1,k}=\big(6(2a-1)^2-1\big)\xi_{1,k-1}+2(2a-1)\big(9(2a-1)^2-3\big)\xi_{2,k-1},$$
$$\xi_{2,k}=2(2a-1)\xi_{1,k-1}+\big(6(2a-1)^2-1\big)\xi_{2,k-1}.$$
\end{lemma}
\begin{proof}
We first prove that $\xi=6(2a-1)^2-1+2(2a-1)\sqrt{9(2a-1)^2-3}$ has norm $1$ in $\mathbb{Z}\big[\sqrt{\big(9(2a-1)^2-3\big)}\big]$. We have 
 $$N(\xi)=\big(6(2a-1)^2-1\big)^2-\big(9(2a-1)^2-3)(2(2a-1)\big)^2=1.$$

Since $x^2-\big(9(2a-1)^2-3\big)y^2=4n-3(2a-1)^2+1$, we have
 \begin{align*}
N\bigg(\big(x+y\sqrt{9(2a-1)^2-3}\big)\xi^k\bigg)&=N\bigg(x+y\sqrt{9(2a-1)^2-3}\bigg)N(\xi^k)\\
 &=4n-3(2a-1)^2+1,
 \end{align*}
 for each integer $k\geq0$. Therefore, setting
 \begin{align}\label{eq 2}
 \xi_{1,k}+\xi_{2,k}\sqrt{9(2a-1)^2-3}=\big(x+y\sqrt{9(2a-1)^2-3}\big)\xi^k,    
 \end{align}
we can observe that $(\xi_{1,k},\xi_{2,k})$ satisfies (\ref{eq pell}). Now, iterating (\ref{eq 2}), for $k\geq1$, we get
 $$\xi_{1,k}=\big(6(2a-1)^2-1\big)\xi_{1,k-1}+2(2a-1)\big(9(2a-1)^2-3\big)\xi_{2,k-1},$$
$$\xi_{2,k}=2(2a-1)\xi_{1,k-1}+\big(6(2a-1)^2-1\big)\xi_{2,k-1},$$
with $\xi_{1,0}=x,\xi_{2,0}=y$. Thus, for each $\xi_{2,k}$, we have some integer $b$ such that $2b-1=\xi_{2,k}$. Consequently, for a fixed $a\geq1$, for each
$$b \in \left\{ \frac{\xi_{2,k}+1}{2} : k\ge 0 \right\},$$
we have $S_aS_b+n=c^2$, where $c$ depends on $b$.
\end{proof}

\begin{lemma}\label{Lemma 3}
Let $a,b,n\geq1$ be three integers with $b>a$. Then $ab+1=S_n$ is equivalent to the generalized Pell  equation
\begin{align}\label{eq star pell}
X^2-6Y^2=(b-a)^2-6, 
\end{align}
where $X=a+b$, $Y=2n-1$.
\end{lemma}
\begin{proof}
Since
$$S_n=6n(n-1)+1=\frac{3(2n-1)^2-1}{2},$$
and $ab+1=S_n$, we have
 $$ab+1=\frac{3(2n-1)^2-1}{2},$$
 which gives $2ab+3=3(2n-1)^2,$ then 
 $$(a+b)^2-6(2n-1)^2=(b-a)^2-6.$$
 \end{proof}
 \begin{lemma}\label{Lemma 4}
 Under the assumptions in Lemma \ref{Lemma 3}, if $(x,y)$ is a positive integer solution of (\ref{eq star pell}), then for each integer $k\geq0$, $(\xi_{1,k},\xi_{2,k})$ is also a positive integer solution of (\ref{eq star pell}), where $\xi_{1,0}=x,\xi_{2,0}=y$ and for $k\geq1$
$$\xi_{1,k}=5\xi_{1,k-1}+12\xi_{2,k-1},\quad \xi_{2,k}=2\xi_{1,k-1}+5\xi_{2,k-1}.$$
\end{lemma}
\begin{proof}
We first prove that $\xi=5+2\sqrt{6}$ has norm $1$ in $\mathbb{Z}[\sqrt{6}]$. We have 
 $$N(\xi)=5^2-6\cdot 2^2=1.$$

Since $x^2-6y^2=(b-a)^2-6$, we have
 \begin{align*}
N\bigg(\big(x+y\sqrt{6}\big)\xi^k\bigg)=N\big(x+y\sqrt{6}\big)N(\xi^k)=(b-a)^2-6,
 \end{align*}
 for each integer $k\geq0$. Therefore, setting
 \begin{align}\label{eq 4}
 \xi_{1,k}+\xi_{2,k}\sqrt{6}=\big(x+y\sqrt{6}\big)\xi^k,    
 \end{align}
we can observe that $(\xi_{1,k},\xi_{2,k})$ satisfies (\ref{eq star pell}). Now, iterating (\ref{eq 4}), for $k\geq1$, we get 
 $$\xi_{1,k}=5\xi_{1,k-1}+12\xi_{2,k-1},\quad \xi_{2,k}=2\xi_{1,k-1}+5\xi_{2,k-1},$$
with $\xi_{1,0}=x,\xi_{2,0}=y$.
\end{proof}
To prove Theorem \ref{Theorem 1.8}, we need the following results due to Nagell, Matveev, and an immediate variation of a result due to Dujella and Peth{\H{o}}.
\begin{theorem}[Nagell, 1951]\label{Nagell}
Let $u+v\sqrt{d}$ be the fundamental solution of the
equation $x^2-dy^2=1$. Then, for each fundamental solution $x^*+y^*\sqrt{d}$ of
the equation $x^2-dy^2=N,$ the following inequalities hold:
$$0\leq y^*\leq\frac{v}{\sqrt{2(u+\epsilon)}}\sqrt{|N|},$$
$$|x^*|\leq\sqrt{\frac{1}{2}(u+\epsilon)|N|},$$
where $\epsilon=1$ if $N>0$, and $\epsilon=-1$ if $N<0$.
In particular, there are only finitely many fundamental solutions
(and therefore, also the classes of solutions).
\end{theorem}
Before stating other two results, let us first recall some basic notions from algebraic number theory.

Let $\eta $ be an algebraic number of degree $d$ with minimal polynomial 
$$a_{0}x^{d}+a_{1}x^{d-1}+\ldots+a_{d}=a_{0}\prod\limits_{i=1}^{d}\left(x-\eta
^{(i)}\right), $$
where the $a_{i}$'s are relatively prime integers with $a_{0}>0$ and $\eta
^{(i)}$'s are conjugates of $\eta $. Then 
\begin{align*}
  h(\eta )=\frac{1}{d}\left( \log a_{0}+\sum\limits_{i=1}^{d}\log \left( \max
\left\{ |\eta ^{(i)}|,1\right\} \right) \right)  
\end{align*}
is called the logarithmic height of $\eta .$ In particular, if $\eta =a/b$
is a rational number with $\gcd (a,b)=1$ and $b>0,$ then $h(\eta )=\log
\left( \max \left\{ |a|,b\right\} \right)$. The following are some well known properties of logarithmic height:
\begin{enumerate}
    \item[\upshape(1)] $h(\eta \pm \gamma )\leq h(\eta )+h(\gamma )+\log 2$,
    \item[\upshape(2)] $h(\eta \gamma ^{\pm 1})\leq h(\eta )+h(\gamma )$,
    \item[\upshape(3)]$h(\eta ^{s})=|s|h(\eta ),~s\in \mathbb{Z}$. 
\end{enumerate}

\begin{theorem}[\cite{r13}]\label{Thm2}
Assume that $\gamma _{1},\gamma _{2},...,\gamma _{r}$ are
positive real algebraic numbers in a real algebraic number field $\mathbb{K}$
of degree $d$, $b_{1},b_{2},\ldots,b_{r}\in \mathbb{Z}$, and 
$$\Lambda :=\gamma _{1}^{b_{1}}...\gamma _{r}^{b_{r}}-1 $$
is not zero. Then 
$$|\Lambda |>\exp \left( -1.4\cdot 30^{r+3}\cdot r^{4.5}\cdot d^{2}\cdot(1+\log
d)\cdot(1+\log B)A_{1}A_{2}...A_{r}\right), $$
where 
$$
B\geq \max \left\{ |b_{1}|,...,|b_{r}|\right\}, 
$$
and 
$$A_{i}\geq \max \left\{ dh(\gamma _{i}),|\log \gamma _{i}|,0.16\right\}\quad \text{for all } i=1,...,r.$$

\end{theorem}

\begin{lemma}\label{Lem1}\cite{r7} Let $M$ be a positive integer, let $p/q$ be a convergent of the
continued fraction expansion of the irrational number $\gamma $ such that $q>6M,$ and
let $A,B,\mu $ be some real numbers with $A>0$ and $B>1.$ Let $\epsilon
:=||\mu q||-M||\gamma q||$, where the function $||\cdot ||$ denotes the
distance from $x$ to the nearest integer, that is, $||x||=\min \left\{
|x-n|:n\in\mathbb{Z}\right\} $ for a real number $x$. If $\epsilon >0,$ then there exists no solution
to the inequality 
$$0<|u\gamma -v+\mu |<AB^{-w}, $$
in positive integers $u,v,$ and $w$ with 
$$u\leq M\text{ and }w\geq \frac{\log (Aq/\epsilon )}{\log B}.$$
\end{lemma}

\section{Proof of the main theorems}

\begin{proof}[Proof of Theorem \ref{Theorem 2}] The proof is divided into the following eight cases depending on the value $k$.

Case-1: Let $k=3$. If $\{S_1,S_b\}=\{1,S_b\}$ is a $D(3)$-pair, then $S_b+3=m^2$, where $m$ is some integer. By Lemma \ref{Lemma 1}, we get $X^2-6Y^2=10$, where $X=2m$, $Y=2b-1$. Since $(X,Y)=(4,1)$ is a solution of $X^2-6Y^2=10$, by Lemma \ref{Lemma 2}, we have infinitely many $b$.

Case-2: Let $k=8$. If $\{S_1,S_b\}=\{1,S_b\}$ is a $D(8)$-pair, then $S_b+8=m^2$, where $m$ is some integer. By Lemma \ref{Lemma 1}, we get $X^2-6Y^2=30$, where $X=2m$, $Y=2b-1$. Since $(X,Y)=(6,1)$ is a solution of $X^2-6Y^2=30$, by Lemma \ref{Lemma 2}, we have infinitely many $b$.

Case-3: Let $k=12$. If $\{S_1,S_b\}=\{1,S_b\}$ is a $D(12)$-pair, then $S_b+12=m^2$, where $m$ is some integer. By Lemma \ref{Lemma 1}, we get $X^2-6Y^2=46$, where $X=2m$, $Y=2b-1$. Since $(X,Y)=(10,3)$ is a solution of $X^2-6Y^2=46$, by Lemma \ref{Lemma 2}, we have infinitely many $b$.

Case-4: Let $k=11$. If possible, assume that $\{S_a,S_b\}$ is a $D(11)$-pair, then $S_aS_b+11=m^2$, where $m$ is some integer. But $S_n\equiv 1,10,13,19\pmod{27}$, which implies that $S_aS_b+11\equiv 3,6,12,15,18,21,24 \pmod{27}$. On the other hand $m^2\equiv 0,1,4,7,9,10,13,16,19,22,25\pmod{27}$. Therefore, $S_aS_b+11$ cannot be a perfect square for any positive integers $a,b$.

Case-5: Let $k=-1$. If $\{S_1,S_b\}=\{1,S_b\}$ is a $D(-1)$-pair, then $S_b-1=m^2$, where $m$ is some integer. By Lemma \ref{Lemma 1}, we get $X^2-6Y^2=-6$, where $X=2m$, $Y=2b-1$. Since $(X,Y)=(12,5)$ is a solution of $X^2-6Y^2=-6$, by Lemma \ref{Lemma 2}, we have infinitely many $b$.

Case-6: Let $k=-4$. If $\{S_1,S_b\}=\{1,S_b\}$ is a $D(-4)$-pair, then $S_b-4=m^2$, where $m$ is some integer. By Lemma \ref{Lemma 1}, we get $X^2-6Y^2=-18$, where $X=2m$, $Y=2b-1$. Since $(X,Y)=(6,3)$ is a solution of $X^2-6Y^2=-18$, by Lemma \ref{Lemma 2}, we have infinitely many $b$.

Case-7: Let $k=-9$. If $\{S_1,S_b\}=\{1,S_b\}$ is a $D(-9)$-pair, then $S_b-9=m^2$, where $m$ is some integer. By Lemma \ref{Lemma 1}, we get $X^2-6Y^2=-38$, where $X=2m$, $Y=2b-1$. Since $(X,Y)=(16,7)$ is a solution of $X^2-6Y^2=-38$, by Lemma \ref{Lemma 2}, we have infinitely many $b$.

Case-8: Let $k=-12$. If $\{S_1,S_b\}=\{1,S_b\}$ is a $D(-12)$-pair, then $S_b-12=m^2$, where $m$ is some integer. By Lemma \ref{Lemma 1}, we get $X^2-6Y^2=-50$, where $X=2m$, $Y=2b-1$. Since $(X,Y)=(10,5)$ is a solution of $X^2-6Y^2=-50$, by Lemma \ref{Lemma 2}, we have infinitely many $b$.
\end{proof}

\begin{proof}[Proof of Theorem \ref{Theorem 3}]
Observe that
\begin{align}\label{eq a,b}
S_aS_b+k=\big(6ab-3(b+a)+1\big)^2-3(b-a)^2+k.    
\end{align}
If $b=a+1$, then we have $3(b-a)^2-k=0$ only for $k=3$. Therefore, $S_aS_{a+1}+k$ is not a perfect square for $k=8,12$ as from (\ref{eq a,b}), we have
$$(6a^2-2)^2<S_aS_{a+1}+k<(6a^2-1)^2,$$
and  $S_aS_{a+1}+k$ is not a perfect square for $k=-1,-4,-9,-12$ as from (\ref{eq a,b}), we have
$$(6a^2-3)^2<S_aS_{a+1}+k<(6a^2-2)^2.$$

For $b=a+2$, we have $3(b-a)^2-k=0$ only for $k=12$. Therefore, $S_aS_{a+2}+k$ is not a perfect square for $k=3,8,-1,-4,-9,-12$ as from (\ref{eq a,b}), we have
$$(6a^2+6a-6)^2<S_aS_{a+2}+k<(6a^2+6a-5)^2.$$
Now, for $b\geq a+3$, if $a=2$, then $4a+2=10$. Then for $5\leq b\leq 9$, it is easy to check that $\{S_2,S_9\}$ is a $D(-4)$-pair and for $k=3,8,12,-1,-9,-12$, there is no  $5\leq b\leq 9$ such that $\{S_2,S_b\}$ is a $D(k)$-pair.\\
Now we claim that if $a+3\leq b\leq 4a+1$ with $a\geq3$, then $\{S_a,S_b\}$ is not a $D(k)$-pair for $k=3,8,12,-1,-4,-9,-12$. To prove the claim, it suffices to prove that $S_aS_b+k$ is not a perfect square for $a+3\leq b\leq 4a+1$ with $a\geq3$ and $k=3,8,12,-1,-4,-9,-12$.

Set $d=b-a$, then $3\leq d\leq 3a+1$. Let us consider the function $f_a:[3,3a+1]\to \mathbb{R}$ defined by $f_a(d)=12ad+12a^2-6d-12a-3d^2+k+1$. Then, since $f''_a(d)=-6<0$ and for $a\geq3$
$$f_a(3)=12a^2+24a+k-44>0,\quad f_a(3a+1)=21a^2-36a+k-8>0,$$
we have $f_a(d)>0$ that is $2\big(6ab-3(a+b)\big)+1>3d^2-k$ for all $3\leq d\leq 3a+1$. It follows from (\ref{eq a,b}) that
$$\big(6ab-3(b+a)\big)^2<S_aS_b+k<\big(6ab-3(b+a)+1\big)^2.$$
This completes the proof.
\end{proof}

\begin{proof}[Proof of Theorem \ref{Theorem 4}] The proof is divided into the following seven cases depending on the value $k$.

Case-1: Let $k=3$. Then for $a\geq2$, we have
$$S_aS_{a+1}+3=(6a^2-2)^2.$$
It follows that $\{S_a,S_{a+1}\}$ is a $D(3)$-pair for $a\geq2$.

Case-2: Let $k=8$. If $\{S_4,S_b\}=\{73,S_b\}$ is a $D(8)$-pair, then $73S_b+8=m^2$, where $m$ is some integer. By Lemma \ref{Lemma 1}, we get $X^2-438Y^2=-114$, where $X=2m$, $Y=2b-1$. Since $(X,Y)=(18,1)$ is a solution of $X^2-438Y^2=-114$, by Lemma \ref{Lemma 2}, we have infinitely many $b$.

Case-3: Let $k=12$. Then for $a\geq2$, we have
$$S_aS_{a+2}+12=(6a^2+6a-5)^2.$$
It follows that $\{S_a,S_{a+2}\}$ is a $D(12)$-pair for $a\geq2$.

Case-4: Let $k=-1$. If $\{S_3,S_b\}=\{37,S_b\}$ is a $D(-1)$-pair, then $37S_b-1=m^2$, where $m$ is some integer. By Lemma \ref{Lemma 1}, we get $X^2-222Y^2=-78$, where $X=2m$, $Y=2b-1$. Since $(X,Y)=(12,1)$ is a solution of $X^2-222Y^2=-78$, by Lemma \ref{Lemma 2}, we have infinitely many $b$.

Case-5: Let $k=-4$. If $\{S_2,S_b\}=\{13,S_b\}$ is a $D(-4)$-pair, then $13S_b-4=m^2$, where $m$ is some integer. By Lemma \ref{Lemma 1}, we get $X^2-78Y^2=-42$, where $X=2m$, $Y=2b-1$. Since $(X,Y)=(6,1)$ is a solution of $X^2-78Y^2=-42$, by Lemma \ref{Lemma 2}, we have infinitely many $b$.

Case-6: Let $k=-9$. If $\{S_2,S_b\}=\{13,S_b\}$ is a $D(-9)$-pair, then $13S_b-9=m^2$, where $m$ is some integer. By Lemma \ref{Lemma 1}, we get $X^2-78Y^2=-62$, where $X=2m$, $Y=2b-1$. Since $(X,Y)=(4,1)$ is a solution of $X^2-78Y^2=-62$, by Lemma \ref{Lemma 2}, we have infinitely many $b$.

Case-7: Let $k=-12$. If $\{S_2,S_b\}=\{13,S_b\}$ is a $D(-12)$-pair, then $13S_b-12=m^2$, where $m$ is some integer. By Lemma \ref{Lemma 1}, we get $X^2-78Y^2=-74$, where $X=2m$, $Y=2b-1$. Since $(X,Y)=(2,1)$ is a solution of $X^2-78Y^2=-74$, by Lemma \ref{Lemma 2}, we have infinitely many $b$.
\end{proof}

\begin{proof}[Proof of Theorem \ref{Theorem 5}]
 Note that
 \begin{align*}
    S_n\pmod{8}=\begin{cases}
        1\quad \text{if}\quad  n\equiv 0,1 \pmod{4},\\
          5\quad \text{if}\quad  n\equiv 2,3 \pmod{4}.
    \end{cases}
 \end{align*}
 Therefore,
  \begin{align*}
    S_nS_m\pmod{8}=\begin{cases}
        1\quad \text{if}\quad  S_n\equiv S_m \pmod{8},\\
          5\quad \text{if}\quad  S_n\not\equiv S_m\pmod{8}.
    \end{cases}
    \end{align*}
Since $k\equiv 12,20 \pmod{24}$, in particular, we have $k\equiv 4 \pmod{8}$. It follows that
     \begin{align*}
    S_nS_m+k\pmod{8}=\begin{cases}
        5\quad \text{if}\quad  S_n\equiv S_m \pmod{8},\\
          1\quad \text{if}\quad  S_n\not\equiv S_m\pmod{8}.
    \end{cases}
    \end{align*}
    Since $m^2\equiv 0,1,4 \pmod{8}$, if $\{S_a,S_b,S_c\}$ is $D(k)$-triple, then we must have $S_i\not\equiv S_j \pmod{8},$
for all $i\neq j$, $i,j\in\{a,b,c\}$, but this is impossible as either $S_a\equiv S_b\pmod{8}$ or  $S_a\equiv S_c\pmod{8}$.\\
Again, observe that
 \begin{align*}
    S_n\pmod{9}=\begin{cases}
        1\quad \text{if}\quad  n\equiv 0,1 \pmod{3},\\
          4\quad \text{if}\quad  n\equiv 2 \pmod{3}.
    \end{cases}
 \end{align*}
 Therefore,
  \begin{align*}
    S_nS_m\pmod{9}=\begin{cases}
     7\quad \text{if}\quad  S_n\equiv S_m\equiv 4 \pmod{9},\\
        1\quad \text{if}\quad  S_n\equiv S_m\equiv 1 \pmod{9},\\
          4\quad \text{if}\quad  S_n\not\equiv S_m\pmod{9}.
    \end{cases}
    \end{align*}
Since $k\equiv 23,32 \pmod{36}$, in particular, we have $k\equiv 5 \pmod{9}$. It follows that
     \begin{align*}
    S_nS_m+k\pmod{9}=\begin{cases}
     3\quad \text{if}\quad  S_n\equiv S_m\equiv 4 \pmod{9},\\
        6\quad \text{if}\quad  S_n\equiv S_m\equiv 1 \pmod{9},\\
          0\quad \text{if}\quad  S_n\not\equiv S_m\pmod{9}.
    \end{cases}
    \end{align*}
    Since $m^2\equiv 0,1,4,7 \pmod{9}$, if $\{S_a,S_b,S_c\}$ is $D(k)$-triple, then we must have $S_i\not\equiv S_j \pmod{9},$
for all $i\neq j$, $i,j\in\{a,b,c\}$, but this is impossible as either $S_a\equiv S_b\pmod{9}$ or  $S_a\equiv S_c\pmod{9}$. This completes the proof.
\end{proof}
\begin{proof}[Proof of Corollary \ref{Corollary 1}]
The proof follows from Theorem \ref{Theorem 5}.
\end{proof}
\begin{proof}[Proof of Theorem \ref{Theorem 6}]
If possible, assume that $\{S_b,S_a,S_{a+1}\}$ is a $D(3)$-triple. Then $\{S_b,S_a\}$ and $\{S_b, S_{a+1}\}$ both are $D(3)$-pair. Set
\begin{align}\label{eq b}
S_bS_a+3=c^2,\quad S_bS_{a+1}+3=d^2,    
\end{align}
where $c,d$ are some integers. Let us consider the following two cases.\\
Case-1: Let $b=1$. Then using (\ref{eq b}) and Lemma \ref{Lemma 1}, we have two corresponding generalized Pell equations
\begin{align}\label{eq b1}
X_1^2-6Y_1^2=10,\quad X_2^2-6Y_2^2=10,    
\end{align}
where $(X_1,Y_1)=(2c,2a-1)$, $(X_2,Y_2)=(2d, 2a+1)$. Observe that $Y_2=Y_1+2$. Since $Y_1\geq 3$, we have
$$10+6Y_1^2<(\sqrt{6}Y_1+1)^2,$$
it follows that
$$\sqrt{6}Y_1<X_1<\sqrt{6}Y_1+1,\quad \sqrt{6}(Y_1+2)<X_2<\sqrt{6}(Y_1+2)+1.$$
Therefore,
$$2\sqrt6-1<X_2-X_1<2\sqrt6+1,$$
that is $X_2-X_1=4$ or $X_2-X_1=5$. Note that $X_2-X_1=5$ is impossible as both $X_i$ are even. If $X_2-X_1=4$, then from (\ref{eq b1}), we get
$$X_1^2-6Y_1^2=(X_1+4)^2-6(Y_1+2)^2,$$
which gives $X_1=3Y_1+1$, that is, $c=3a-1$. Then from
$$S_a+3=6a(a-1)+4=(3a-1)^2,$$
we get $a=1$, which contradicts the fact that $a\geq2$.\\
Case-2: Let $b\geq2$ and set $D=9(2b-1)^2-3$, $R=3(2b-1)^2-13$. Since $a>b$, we can assume that $a\geq b+2$ as if $a=b+1$, then $\{S_b,S_{b+1},S_{b+2}\}$ cannot be a $D(3)$-triple by Theorem \ref{Theorem 3}. Now, from (\ref{eq b}), we have
\begin{align}\label{eq b>1}
 X_1^2-DY_1^2=-R,\quad  X_2^2-DY_2^2=-R,  
\end{align}
where $(X_1,Y_1)=(2c,2a-1),(X_2,Y_2)=(2d,2a+1)$. Again, we have that $Y_2=Y_1+2$. For $b\geq2$, we have $R>0$. Therefore, for $i=1,2$, we have
$$X_i^2-(\sqrt{D}Y_i)^2=-R,$$
which gives
\begin{align}\label{eq i}
X_i=\sqrt{D}Y_i-\frac{R}{\sqrt{D}Y_i+X_i}.    
\end{align}
Since $Y_2>Y_1$ and $X_2>X_1$, we get
\begin{align}\label{inequality}
\sqrt{D}Y_2+X_2>\sqrt{D}Y_1+X_1.   
\end{align}
Using (\ref{eq i}), and the fact that $Y_2-Y_1=2$, we get
\begin{align}\label{diff}
 X_2-X_1=2\sqrt{D}+R\big(\frac{1}{\sqrt{D}Y_1+X_1}-\frac{1}{\sqrt{D}Y_2+X_2}\big),   
\end{align}

using (\ref{inequality}), and the above equation, we get
\begin{align}\label{first}
X_2-X_1>2\sqrt{D}.    
\end{align}
For $b\geq2$, we have
$$\big(3(2b-1)-1\big)^2=9(2b-1)^2-6(2b-1)+1<9(2b-1)^2-3=D<\big(3(2b-1)\big)^2,$$
from which we obtain
\begin{align}\label{bound}
6(2b-1)-2<2\sqrt{D}<6(2b-1).    
\end{align}
Now, using (\ref{first}) and (\ref{bound}), we get
\begin{align}\label{first 1}
  6(2b-1)-2<X_2-X_1.  
\end{align}
Since $a\geq b+2$ and $Y_1=2a-1$, we have $Y_1=2a-1\geq 2b+2$. Therefore, from (\ref{diff}), we have
$$R\big(\frac{1}{\sqrt{D}Y_1+X_1}-\frac{1}{\sqrt{D}Y_2+X_2}\big)<\frac{R}{\sqrt{D}Y_1+X_1}<\frac{3(2b-1)^2}{\big(3(2b-1)-1\big)(2b+2)}\leq 1,$$
which together with (\ref{diff}) imply that
$$X_2-X_1<2\sqrt{D}+1<6(2b-1)+1.$$
Consequently, from (\ref{first 1}), we obtain
$$6(2b-1)-2<X_2-X_1<2\sqrt{D}+1<6(2b-1)+1,$$
it follows that $X_2-X_1=6(2b-1)$ as $X_2-X_1$ is even. Therefore, from (\ref{eq b>1}), we get
$$X_1^2-DY_1^2=(6(2b-1)+X_1)^2-D(Y_1+2)^2,$$
which gives
$$3(2b-1)X_1+9(2b-1)^2=D(Y_1+1)=2aD.$$
Therefore
\begin{align*}
X_1=\frac{2Da-9(2b-1)^2}{3(2b-1)}=\frac{2a\big(3(2b-1)^2-1\big)-3(2b-1)^2}{2b-1},
\end{align*}
since $X_1=2c$, we must have $X_1\equiv 0 \pmod{2}$, but
$$2a\big(3(2b-1)^2-1\big)-3(2b-1)^2\not\equiv 0\pmod{2},$$
which is absurd.\\
Hence $\{S_b,S_a,S_{a+1}\}$ cannot be a $D(3)$-triple. This completes the proof.
\end{proof}
\begin{proof}[Proof of Theorem \ref{Theorem a+1}]
 If $\{a,a+1, 12(2a+1)\}$ is a star Diophantine $D(1)$-triple, then
 $$a(a+1)+1=S_n,\quad 12a(2a+1)+1=S_m,\quad 12(a+1)(2a+1)+1=S_r,$$
 where $n,m,r$ are some integers. Since $12a(2a+1)+1=S_m$ implies $m=2a+1$, and $12(a+1)(2a+1)+1=S_r$ implies  $r=2a+2$, it is sufficient to find $a,n$ such that $a(a+1)+1=S_n$. By Lemma \ref{Lemma 3}, $a(a+1)+1=S_n$ is equivalent to
$$X^2-6Y^2=-5,$$
where $X=2a+1$ and $Y=2n-1$. Observe that $(x,y)=(17,7)$ is a solution of $X^2-6Y^2=-5$. Therefore, by Lemma \ref{Lemma 4}, $(\xi_{1,k},\xi_{2,k})$ is also a positive integer solution of $X^2-6Y^2=-5$, where $\xi_{1,0}=17,\xi_{2,0}=7$ and for $k\geq1$
$$\xi_{1,k}=5\xi_{1,k-1}+12\xi_{2,k-1},\quad \xi_{2,k}=2\xi_{1,k-1}+5\xi_{2,k-1}.$$
Since $X=2a+1$, for each
$$a\in\bigg\{ \frac{5\xi_{1,k-1}+12\xi_{2,k-1}-1}{2}:k\geq 1,~~\xi_{1,0}=17,\xi_{2,0}=7\bigg\},$$
$\{a,a+1, 12(2a+1)\}$ is a star Diophantine $D(1)$-triple.\\
If possible, assume that $\{a,a+1,12(2a+1),b\}$ is a star Diophantine $D(1)$-quadruple with $12\nmid b$. Then, we have some integers $n,m$ such that
$$ab+1=S_n=6n(n-1)+1,\quad (a+1)b+1=S_m=6m(m-1)+1,$$
since $6n(n-1), 6m(m-1)\equiv 0 \pmod{12}$, we must have $12\mid b$, but this contradicts the fact that $12\nmid b$. Hence $\{a,a+1,12(2a+1),b\}$ cannot be a star Diophantine $D(1)$-quadruple with $12\nmid b$. This completes the proof.
\end{proof}
 \begin{proof}[Proof of Theorem \ref{Theorem a+2}]
 If $\{a,a+2,6(a+1)\}$ is a star Diophantine $D(1)$-triple, then
 $$a(a+2)+1=S_n,\quad 6a(a+1)+1=S_m,\quad 6(a+2)(a+1)+1=S_r,$$
 where $n,m,r$ are some integers. Since $6a(a+1)+1=S_m$ implies $m=a+1$, and $6(a+2)(a+1)+1=S_r$ implies  $r=a+2$, it is sufficient to find $a,n$ such that $a(a+2)+1=S_n$. By Lemma \ref{Lemma 3}, $a(a+2)+1=S_n$ is equivalent to
$$X^2-6Y^2=-2,$$
where $X=2a+2$ and $Y=2n-1$. Observe that $(x,y)=(22,9)$ is a solution of $X^2-6Y^2=-2$. Therefore, by Lemma \ref{Lemma 4}, $(\xi_{1,k},\xi_{2,k})$ is also a positive integer solution of $X^2-6Y^2=-2$, where $\xi_{1,0}=22,\xi_{2,0}=9$ and for $k\geq1$
$$\xi_{1,k}=5\xi_{1,k-1}+12\xi_{2,k-1},\quad \xi_{2,k}=2\xi_{1,k-1}+5\xi_{2,k-1}.$$
Since $X=2a+2$, for each
$$a\in \bigg\{\frac{5\xi_{1,k-1}+12\xi_{2,k-1}-2}{2}:k\geq1,~~\xi_{1,0}=22,\xi_{2,0}=9\bigg\},$$
$\{a,a+2,6(a+1)\}$ is a star Diophantine $D(1)$-triple.\\
If possible, assume that $\{a,a+2,6(a+1),b\}$ is a star Diophantine $D(1)$-quadruple with $6\nmid b$. Then, we have some integers $n,m$ such that
$$ab+1=S_n,\quad (a+2)b+1=S_m,$$
since $6n(n-1), 6m(m-1)\equiv 0 \pmod{12}$, we must have $6\mid b$, but this contradicts the fact that $6\nmid b$. Hence $\{a,a+2,6(a+1),b\}$ cannot be a star Diophantine $D(1)$-quadruple with $6\nmid b$. This completes the proof.
\end{proof}

 \begin{proof}[Proof of Theorem \ref{Theorem a+3}]
 If $\{a,a+3,\dfrac{8a+12}{3}\}$ is a star Diophantine $D(1)$-triple, then
 $$a(a+3)+1=S_n,\quad \frac{a(8a+12)}{3}+1=S_m,\quad \frac{(a+3)(8a+12)}{3}+1=S_r,$$
 where $n,m,r$ are some integers. Since $\dfrac{a(8a+12)}{3}+1=S_m$ implies $m=\dfrac{2a+3}{3}$, and $\dfrac{(a+3)(8a+12)}{3}+1=S_r$ implies  $r=\dfrac{2a+6}{3}$, it is sufficient to find $a,n$ such that $a(a+3)+1=S_n$ with $3\mid a$. By Lemma \ref{Lemma 3}, $a(a+3)+1=S_n$ is equivalent to
$$X^2-6Y^2=3,$$
where $X=2a+3$ and $Y=2n-1$. Observe that $(x,y)=(27,11)$ is a solution of $X^2-6Y^2=3$. Therefore, by Lemma \ref{Lemma 4}, $(\xi_{1,k},\xi_{2,k})$ is also a positive integer solution of $X^2-6Y^2=3$, where $\xi_{1,0}=27,\xi_{2,0}=11$ and for $k\geq1$
$$\xi_{1,k}=5\xi_{1,k-1}+12\xi_{2,k-1},\quad \xi_{2,k}=2\xi_{1,k-1}+5\xi_{2,k-1}.$$
Since $X=2a+3$, and $3\mid X^2-6Y^2$, we have $3\mid a$. Hence for each
$$a\in \bigg\{\frac{5\xi_{1,k-1}+12\xi_{2,k-1}-3}{2}:k\geq1,~~\xi_{1,0}=27,\xi_{2,0}=11\bigg\},$$
$\{a,a+3,\dfrac{8a+12}{3}\}$ is a star Diophantine $D(1)$-triple.\\
If possible, assume that $\{a,a+3,\dfrac{8a+12}{3},b\}$ is a star Diophantine $D(1)$-quadruple with $4\nmid b$. Then, we have some integers $n,m$ such that
$$ab+1=S_n,\quad (a+3)b+1=S_m,$$
since $6n(n-1), 6m(m-1)\equiv 0 \pmod{12}$, we must have $4\mid b$, but this contradicts the fact that $4\nmid b$. Hence $\{a,a+3,\dfrac{8a+12}{3},b\}$ cannot be a star Diophantine $D(1)$-quadruple with $4\nmid b$. This completes the proof.
\end{proof}

\begin{proof}[Proof of Theorem \ref{Theorem a+4}]
If $\{a,a+4,\dfrac{3a+6}{2}\}$ is a star Diophantine $D(1)$-triple, then
 $$a(a+4)+1=S_n,\quad \frac{3a(a+2)}{2}+1=S_m,\quad \frac{3(a+4)(a+2)}{2}+1=S_r,$$
 where $n,m,r$ are some integers. Since $\dfrac{3a(a+2)}{2}+1=S_m$ implies $m=\dfrac{a+2}{2}$, and $\dfrac{3(a+4)(a+2)}{2}+1=S_r$ implies  $r=\dfrac{a+4}{2}$, it is sufficient to find $a,n$ such that $a(a+4)+1=S_n$ with $2\mid a$. By Lemma \ref{Lemma 3}, $a(a+4)+1=S_n$ is equivalent to
$$X^2-6Y^2=10,$$
where $X=2a+4$ and $Y=2n-1$. Observe that $(x,y)=(8,3)$ is a solution of $X^2-6Y^2=10$. Therefore, by Lemma \ref{Lemma 4}, $(\xi_{1,k},\xi_{2,k})$ is also a positive integer solution of $X_1^2-6Y_1^2=10$, where $\xi_{1,0}=8,\xi_{2,0}=3$ and for $k\geq1$
$$\xi_{1,k}=5\xi_{1,k-1}+12\xi_{2,k-1},\quad \xi_{2,k}=2\xi_{1,k-1}+5\xi_{2,k-1}.$$
Since $X=2a+4$, and $4\mid 10+6Y^2=X^2$, we have $2\mid a$. Hence for each
$$a\in \bigg\{\frac{5\xi_{1,k-1}+12\xi_{2,k-1}-4}{2}: k\geq 1,~~\xi_{1,0}=8,\xi_{2,0}=3\bigg\},$$
$\{a,a+4,\dfrac{3a+6}{2}\}$ is a star Diophantine $D(1)$-triple.\\
If possible, assume that $\{a,a+4,\dfrac{3a+6}{2},b\}$ is a star Diophantine $D(1)$-quadruple with $3\nmid b$. Then, we have some integers $n,m$ such that
$$ab+1=S_n,\quad (a+4)b+1=S_m,$$
Since $6n(n-1), 6m(m-1)\equiv 0 \pmod{12}$, we must have $3\mid b$, but this contradicts the fact that $3\nmid b$. Hence $\{a,a+4,\dfrac{3a+6}{2},b\}$ cannot be a star Diophantine $D(1)$-quadruple with $3\nmid b$. This completes the proof.
\end{proof}
\begin{proof}[Proof of Theorem \ref{Theorem 1.7}]
If $\{4,108,a\}$ is a star Diophantine $D(1)$-triple, then we have some integers $n,m$ such that
 $$4a+1=S_n,\quad 108a+1=S_m.$$
 Using Lemma \ref{Lemma 3}, the above equations can be written as
 $$X_1^2-6Y_1^2=(a-4)^2-6,\quad X_2^2-6Y_2^2=(a-108)^2-6,$$
 where $X_1=a+4$, $Y_1=2n-1$ and $X_2=a+108$, $Y_2=2m-1$. Then
 \begin{align}\label{eq for a}
 a=\frac{3Y_1^2-3}{8},\quad a=\frac{Y_2^2-1}{72}.    
 \end{align}
From the above equations, we get
 \begin{align*}
 24a+9=9Y_1^2,\quad 72a+1=Y_2^2,    
 \end{align*}
it follows that
 \begin{align}\label{eq 15}
 Y_2^2-27Y_1^2=-26.    
 \end{align}

Since $\alpha=26+5\sqrt{27}$ is a fundamental solution of the equation $x^2-27y^2=1$, by Theorem \ref{Nagell}, for each fundamental solution $x^*+y^*\sqrt{27}$ of (\ref{eq 15}) satisfies
$$0\leq y^*\leq\frac{5}{\sqrt{2(26-1)}}\sqrt{26}<4,$$
$$|x^*|\leq\sqrt{\frac{1}{2}(26-1)26}<19.$$
It is easy to check that (\ref{eq 15}) is satisfied only for $(x^*,y^*)=(\pm1,1)$. It follows that the fundamental solutions of (\ref{eq 15}) are $(1,1)$ and $(-1,1)$. 
 Since $26+5\sqrt{27}$ has norm $1$ in $\mathbb{Z}[\sqrt{27}]$, we can observe that  for $k\geq0$, $(\xi_{1,k},\xi_{2,k})$ and $(\xi_{1,k}',\xi_{2,k}')$ satisfy (\ref{eq 15}), where
\begin{align}\label{eq 5}
 \xi_{1,k}+\xi_{2,k}\sqrt{27}=\big(-1+\sqrt{27}\big)(26+5\sqrt{27})^k,
    \end{align}
 and
 \begin{align}\label{eq 5 new}
\xi_{1,k}'+\xi_{2,k}'\sqrt{27}=\big(1+\sqrt{27}\big)(26+5\sqrt{27})^k.    
 \end{align}
Now, iterating (\ref{eq 5}) and (\ref{eq 5 new}), for $k\geq1$, we get 
 $$\xi_{1,k}=26\xi_{1,k-1}+135\xi_{2,k-1},\quad \xi_{2,k}=5\xi_{1,k-1}+26\xi_{2,k-1},$$
with $\xi_{1,0}=-1,\xi_{2,0}=1$, and
$$\xi_{1,k}'=26\xi_{1,k-1}'+135\xi_{2,k-1}',\quad \xi_{2,k}'=5\xi_{1,k-1}'+26\xi_{2,k-1}',$$
with $\xi_{1,0}'=1,\xi_{2,0}'=1$. Therefore, we have 
$$Y_1\in A_1=\{\xi_{2,k}: k\geq 0,~~\xi_{1,0}=-1,\xi_{2,0}=1\}$$
where $\xi_{1,k}=26\xi_{1,k-1}+135\xi_{2,k-1}, \xi_{2,k}=5\xi_{1,k-1}+26\xi_{2,k-1},$ and
$$Y_1\in A_2=\{\xi'_{2,k}: k\geq 0,~~\xi'_{1,0}=1,\xi'_{2,0}=1\},$$
where $\xi'_{1,k}=26\xi'_{1,k-1}+135\xi'_{2,k-1},\quad \xi'_{2,k}=5\xi'_{1,k-1}+26\xi'_{2,k-1}.$
Since from (\ref{eq for a}), we have $a=\dfrac{3Y_1^2-3}{8}$, and $\xi_{2,0},\xi_{2,0}'=1$ gives $a=0$, it follows that for each
$$a \in \bigg\{ \frac{3\left(\xi_{2,k}^{2}-1\right)}{8} : k\ge 1, \xi_{1,0}=-1, \xi_{2,0}=1 \bigg\}
\cup
\bigg\{ \frac{3\left(\xi_{2,k}^{2}-1\right)}{8} : k\ge 1, \xi_{1,0}=1, \xi_{2,0}=1 \bigg\},$$
where
$$\xi_{1,k}=26\,\xi_{1,k-1}+135\,\xi_{2,k-1},\quad
\xi_{2,k}=5\,\xi_{1,k-1}+26\,\xi_{2,k-1},$$
$\{4,108,a\}$ is a star Diophantine $D(1)$-triple. This completes the proof.
\end{proof}
\begin{proof}[Proof of Theorem \ref{Theorem 1.8}] If $\{4,108,165,a\}$ is a star Diophantine $D(1)$-quadruple, then we have some integers $n,m,r$ such that
$$4a+1=S_n,\quad 108a+1=S_m,\quad 165a+1=S_r.$$
Then, using Lemma \ref{Lemma 3}, the above three equations can be written as
$$X_1^2-6Y_1^2=(a-4)^2-6,\quad X_2^2-6Y_2^2=(a-108)^2-6,\quad X_3^2-6Y_3^2=(a-165)^2-6,$$
where $X_1=a+4$, $Y_1=2n-1$, $X_2=a+108$, $Y_2=2m-1$, $X_3=a+165$, $Y_3=2r-1$. Then
  $$a=\frac{3Y_1^2-3}{8},\quad a=\frac{Y_2^2-1}{72},\quad a=\frac{Y_3^2-1}{110}.$$  
Therefore, from the above three equations, we obtain
$$Y_2^2-27Y_1^2=-26,\quad 55Y_2^2-36Y_3^2=19.$$
Let us consider the following two cases.\\
Case-1: For the generalized Pell equation $Y_2^2-27Y_1^2=-26$, from the proof of Theorem \ref{Theorem 1.7}, we get
$$Y_2\in A_1=\{\xi_{1,k}: k\geq 0,~~\xi_{1,0}=-1,\xi_{2,0}=1\}$$
where $\xi_{1,k}=26\xi_{1,k-1}+135\xi_{2,k-1}, \xi_{2,k}=5\xi_{1,k-1}+26\xi_{2,k-1},$ and
$$Y_2\in A_2=\{\xi'_{1,k}: k\geq 0,~~\xi'_{1,0}=1,\xi'_{2,0}=1\},$$
where $\xi'_{1,k}=26\xi'_{1,k-1}+135\xi'_{2,k-1},\quad \xi'_{2,k}=5\xi'_{1,k-1}+26\xi'_{2,k-1}.$
Since $Y_2=2m-1\geq1$, we exclude $\xi_{1,0}$ and consider
$$Y_2\in A_1=\{\xi_{1,k}: k\geq 1,~~\xi_{1,0}=-1,\xi_{2,0}=1\}.$$
Observe that for $k\geq0$, we have
$$\xi_{1,k+2}=52\xi_{1,k+1}-\xi_{1,k},\quad \xi'_{1,k+2}=52\xi'_{1,k+1}-\xi'_{1,k}.$$ 
Therefore,
 \begin{align*}
   \xi_{1,k} \pmod{8}=\begin{cases}
        5\quad \text{if}\quad  k\equiv 1,2 \pmod{4},\\
          7\quad \text{if}\quad  k\equiv 0,3 \pmod{4}.
    \end{cases}
 \end{align*}

Also, note that for $k\geq0$, we have
$$\xi'_{1,k}=\frac{(1+\sqrt{27})\alpha^{k}+(1-\sqrt{27})\bar{\alpha}^{k}}{2},\quad \text{where}\quad \alpha=26+5\sqrt{27}.$$
Case-2: The generalized Pell equation $55Y_2^2-36Y_3^2=19$ can be written as $Y_4^2-1980Y_2^2=-684$, where $Y_4=36Y_3$. Since $\beta=89+2\sqrt{1980}$ is a fundamental solution of the equation $x^2-1980y^2=1$, by Theorem \ref{Nagell}, for each fundamental solution $x^*+y^*\sqrt{1980}$ of $x^2-1980y^2=-684$ satisfies
$$0\leq y^*\leq\frac{2}{\sqrt{2(89-1)}}\sqrt{684}<4,$$
$$|x^*|\leq\sqrt{\frac{1}{2}(89-1)684}<174.$$

It is easy to check that $x^2-1980y^2=-684$ is satisfied only for $(x^*,y^*)=(\pm 36,1)$. Since $89+2\sqrt{1980}$ has norm $1$ in $\mathbb{Z}[\sqrt{1980}]$, we can observe that  for $k\geq0$, $(\eta_{1,k},\eta_{2,k})$ and $(\eta_{1,k}',\eta_{2,k}')$ satisfy $Y_4^2-1980Y_2^2=-684$, where
\begin{align}\label{eq 19}
 \eta_{1,k}+\eta_{2,k}\sqrt{1980}=\big(-36+\sqrt{1980}\big)(89+2\sqrt{1980})^k,
    \end{align}
 and
 \begin{align}\label{eq 20}
 \eta_{1,k}'+\eta_{2,k}'\sqrt{1980}=\big(36+\sqrt{1980}\big)(89+2\sqrt{1980})^k.    
 \end{align}
Now, iterating (\ref{eq 19}) and (\ref{eq 20}), for $k\geq1$, we get 
 $$\eta_{1,k}=89\eta_{1,k-1}+3960\eta_{2,k-1},\quad \eta_{2,k}=2\eta_{1,{k-1}}+89\eta_{2,{k-1}},$$
with $\eta_{1,0}=-36,\eta_{2,0}=1$, and
$$\eta_{1,k}'=89\eta_{1,k-1}'+3960\eta_{2,k-1}',\quad \eta_{2,k}'=2\eta_{1,{k-1}}'+89\eta_{2,{k-1}}',$$
with $\eta_{1,0}'=36,\eta_{2,0}'=1$. Therefore, we have 
$$Y_2\in B_1=\{\eta_{2,k}: k\geq 0,~~\eta_{1,0}=-36,\eta_{2,0}=1\},$$
where
$$\eta_{1,k}=89\eta_{1,k-1}+3960\eta_{2,k-1},\quad \eta_{2,k}=2\eta_{1,{k-1}}+89\eta_{2,{k-1}}.$$
Now, using the above two relations, we get
$$\eta_{2,k+2}=178\eta_{2,k+1}-\eta_{2,k},$$
thus, for $k\geq1$, we have $\eta_{2,k} \equiv 1\pmod{8}$. And
$$ Y_2\in B_2=\{\eta'_{2,k}: k\geq 0,~~\eta'_{1,0}=36,\eta'_{2,0}=1\},$$
where
 $$\eta'_{1,k}=89\eta'_{1,k-1}+3960\eta'_{2,k-1},\quad \eta'_{2,k}=2\eta'_{1,{k-1}}+89\eta'_{2,{k-1}}.$$
Now, using the above two relations, we get
 $$\eta'_{2,k+2}=178\eta'_{2,k+1}-\eta'_{2,k},$$
 thus, for $k\geq1$, we have  $\eta'_{2,k} \equiv 1\pmod{8}$. Therefore, $A_1\cap(B_1\cup B_2)=\emptyset$. Also, for $k\geq0$, we have
$$\xi_{2,k}=\frac{(-36+\sqrt{1980})\beta^{k}+(36+\sqrt{1980})\bar{\beta}^{k}}{2\sqrt{1980}},\quad \xi'_{2,k}=\frac{(36+\sqrt{1980})\beta^k-(36-\sqrt{1980})\bar{\beta}^k}{2\sqrt{1980}},$$
where $\beta=89+2\sqrt{1980}$.

First few elements of $A_2$, $B_1$, $B_2$ are the following:
$$A_2=\{1,161, 8371, 435131, 22618441, 1175723801,\ldots\},$$
$$B_1=\{1,17, 3025, 538433, 95838049, 17058634289,\ldots\},$$
$$B_2=\{1,161, 28657, 5100785, 907911073, 161603070209,\ldots\}.$$
Now, we prove that $A_2\cap B_2=\{1,161\}$ and $A_2\cap B_1=\{1\}$. Since $Y_2=1$ is impossible as if $Y_2=1$, then $a=0$, consequently, we get the unique $a=360$ corresponding to the value $Y_2=161$.
Set
$$C=(1+\sqrt{27})\sqrt{1980},\quad C'= (1-\sqrt{27})\sqrt{1980}, \quad D=(36+\sqrt{1980}),\quad D'=(36-\sqrt{1980}).$$
Consider the following two cases.\\
Case-3: Since $Y_2\neq 1$, we assume that $\xi'_{1,n}=\eta'_{2,m}$ for some $n,m\geq2$. Since $\beta^{n/2}<\xi'_{1,n}=\xi'_{2,m}<\beta^m$, we have $n<2m$.
Then
$$\frac{(1+\sqrt{27})\alpha^{n}+(1-\sqrt{27})\bar{\alpha}^{n}}{2}=\frac{D\beta^m-D'\bar{\beta}^m}{2\sqrt{1980}},$$
which implies that
$$C\alpha^n-D\beta^m=-C'\bar{\alpha}^n-D'\bar{\beta}^m.$$
Since $0<\bar{\alpha}^n<1$ and $0<\bar{\beta}^m<1$, it follows that
$$|C\alpha^n-D\beta^m|\leq |C'|\bar{\alpha}^n+|D'|\bar{\beta}^m<196.$$
Then
\begin{align}\label{eq 6}
    \left|\frac{C\alpha^n}{D\beta^m}-1\right|<\frac{196}{D\beta^m}<\frac{3}{\beta^m}.
\end{align}
Now, let us apply Theorem \ref{Thm2} with $(\gamma _{1},b_1)=(\dfrac{C}{D},1),~(\gamma _{2},b_2)=(\alpha ,n),~(\gamma_{3},b_3)=(\beta,-m)$, and 
$$\Lambda = \frac{C\alpha^n}{D\beta^m}-1.$$
We claim that $\Lambda\neq0$ for $n,m\geq2$. If possible, assume that $\Lambda=0$ for some $n,m\geq2$. Then $C\alpha^n=D\beta^m$. Since $\alpha=26+5\sqrt{27}$, we have $\alpha^n\in \mathbb{Q}(\sqrt{3})$ and since $\beta=89+12\sqrt{55}$, we have $\beta^m\in \mathbb{Q}(\sqrt{55})$. Now $C\alpha^n=D\beta^m,$ implies
$$(1+\sqrt{27})\alpha^n=\frac{D\beta^m}{6\sqrt{55}}\in \mathbb{Q}(\sqrt{55}),$$
consequently, we get $(1+\sqrt{27})\alpha^n\in \mathbb{Q}(\sqrt{3})\cap\mathbb{Q}(\sqrt{55})=\mathbb{Q}$, which is a contradiction.

Note that $\gamma_1,\gamma_2,\gamma_3\in \mathbb{Q}(\sqrt3,\sqrt{55})$, we set $d=[\mathbb{Q}(\sqrt3,\sqrt{55}):\mathbb{Q}]=4$. Set
$$B=2m\geq\max\{|b_1|,|b_2|,|b_3|\},$$
$$A_2=8>2\log(\alpha)\geq \max\{D h(\gamma_2),\,|\log \gamma_2|,\,0.16\}$$
$$A_3=11>2\log(\beta)\geq \max\{D h(\gamma_3),\,|\log \gamma_3|,\,0.16\}.$$
Since $h(\gamma_1)\leq h(C)+h(D)=\frac{1}{2}(\log51480+\log684)<9$, we set
$$A_1=36>\max\{D h(\gamma_1),\,|\log \gamma_1|,\,0.16\}.$$
From Theorem \ref{Thm2} and (\ref{eq 6}), we obtain
\begin{align*}
    \exp\big(-1.4\cdot 30^{6}\cdot 3^{4.5}\cdot 16\cdot(1+\log 4)\cdot A_1\cdot A_2\cdot A_3\cdot (1+\log2m)\big)<|\Lambda|<\frac{3}{\beta^{m}},
\end{align*}
it follows that
\begin{align*}
    -1.4\cdot 30^{6}\cdot 3^{4.5}\cdot 16\cdot(1+\log 4)\cdot A_1\cdot A_2\cdot A_3\cdot (1+\log2m)<\log3-m\log\beta,
\end{align*}
which gives
\begin{align*}
   m\log\beta&<\log3+1.4\cdot 30^{6}\cdot 3^{4.5}\cdot 16\cdot(1+\log 4)\cdot A_1\cdot A_2\cdot A_3\cdot (1+\log2m)\\
   &<2\cdot 10^{16}\cdot (1+\log2m).
\end{align*}
The above inequality implies that $m<1.6\cdot 10^{17}$. We now need to reduce the upper bound obtained for $n$. For this, we use Lemma \ref{Lem1}.

Let $z=\log(\dfrac{C}{D})+n\log\alpha-m\log\beta$. Observe that $z\neq0$ as $\Lambda\neq0$ for $n,m\geq2$. If $z>0$, then from (\ref{eq 6}), we have
$$0<z\leq e^{z}-1<\frac{3}{\beta^m}.$$
If $z<0$, then from (\ref{eq 6}), we have
$$|e^{z}-1|<\frac{3}{\beta^m}<\frac{1}{2}, \quad \text{for }m\geq1.$$
Since for any nonzero real number $x<0$, if $|e^{x}-1| <\dfrac{1}{2}$, we have $
|x| < 2|e^{x}-1|$. Therefore, for $m\geq1$, we have
$$|z|<2|e^z-1|<\frac{6}{\beta^m}.$$
From the above inequality, we get
\begin{align*}
\left|\log(\frac{C}{D})+n\log\alpha-m\log\beta\right| <\frac{6}{\beta^m},  
\end{align*}
After multiplying both sides of the above inequality by $\dfrac{1}{\log\beta}$, we obtain
\begin{align*}
\left|\frac{\log(\dfrac{C}{D})}{\log\beta}+n\frac{\log\alpha}{\log\beta}-m\right| <\frac{6}{\log\beta\cdot\beta^m}<\frac{2}{\beta^m}.  
\end{align*}
Now, let us apply Lemma \ref{Lem1} with
\begin{align*}
u=n,\quad \gamma=\frac{\log\alpha}{\log \beta},\quad v=m,\quad \mu=\frac{\log(\dfrac{C}{D})}{\log\beta},\quad A=2,\quad B=\beta,\quad w=m.
\end{align*}
Since $m<1.6\cdot 10^{17}$, and $u=n<2m<3.2\cdot 10^{17}$, we set $M=3.2\cdot 10^{17}$. Choosing  
$$\frac{p}{q}=\frac{p_{39}}{q_{39}}=\frac{446050159788405153046}{585015235298117452451},$$
we get $q_{39}>6M$ and
$$\epsilon=||\mu q||-M||\gamma q||=0.1886\ldots>0.$$
Therefore, by Lemma \ref{Lem1}, we have
\begin{align*}
    m<\frac{\log(Aq/\epsilon)}{\log B}<10.
\end{align*}
Therefore, $m\leq 9$. Now, for $2\leq m\leq 9$, we have $\xi'_{1,2}=\eta'_{2,2}=161$.\\
Case-4: Since $Y_2\neq 1$, we assume that $\xi'_{1,n}=\eta_{2,m}$ for some $n,m\ge2$.
Since $\beta^{n/2}<\xi'_{1,n}=\xi_{2,m}<\beta^m$, we have $n<2m$. Then
$$\frac{(1+\sqrt{27})\alpha^{n}+(1-\sqrt{27})\bar{\alpha}^{n}}{2}=\frac{-D'\beta^m+D\bar{\beta}^m}{2\sqrt{1980}},$$
which implies that
$$C\alpha^n+D'\beta^m=-C'\bar{\alpha}^n+D\bar{\beta}^m,$$

Since $0<\bar{\alpha}^n<1$ and $0<\bar{\beta}^m<1$, it follows that
$$|C\alpha^n+D'\beta^m|\leq |C'|\bar{\alpha}^n+D\bar{\beta}^m<268.$$
Then
\begin{align}\label{eq 10}
\left|\frac{C\alpha^n}{(-D')\beta^m}-1\right|<\frac{268}{|D'|\beta^m}<\frac{32}{\beta^m}.
\end{align}
Now, let us apply Theorem \ref{Thm2} with $(\gamma _{1},b_1)=(\dfrac{C}{-D'},1),~(\gamma _{2},b_2)=(\alpha ,n),~(\gamma_{3},b_3)=(\beta,-m)$ and
$$\Lambda = \frac{C\alpha^n}{(-D')\beta^m}-1 .$$
A similar argument to that used in Case-3 proves that $\Lambda\neq 0$.
Note that $\gamma_1,\gamma_2,\gamma_3\in \mathbb{Q}(\sqrt3,\sqrt{55})$, we set $d=[\mathbb{Q}(\sqrt3,\sqrt{55}):\mathbb{Q}]=4$. Set
$$B=2m\geq\max\{|b_1|,|b_2|,|b_3|\},\quad A_2=8>2\log(\alpha)\ge \max\{D h(\gamma_2),|\log\gamma_2|,0.16\},$$
$$A_3=11>2\log(\beta)\ge \max\{D h(\gamma_3),|\log \gamma_3|,0.16\}.$$
Since $h(D')=h(D)$, we have $h(\gamma_1)\leq h(C)+h(D')=h(C)+h(D)<9$ as before, so we set $A_1=36$.
From Theorem \ref{Thm2} and (\ref{eq 10}), we obtain
\begin{align*}
    \exp\big(-1.4\cdot 30^{6}\cdot 3^{4.5}\cdot 16\cdot(1+\log 4)\cdot A_1\cdot A_2\cdot A_3\cdot (1+\log2m)\big)<|\Lambda|<\frac{32}{\beta^{m}},
\end{align*}
it follows that
\begin{align*}
    -1.4\cdot 30^{6}\cdot 3^{4.5}\cdot 16\cdot(1+\log 4)\cdot A_1\cdot A_2\cdot A_3\cdot (1+\log2m)<\log32-m\log\beta,
\end{align*}
which gives
\begin{align*}
   m\log\beta&<\log32+1.4\cdot 30^{6}\cdot 3^{4.5}\cdot 16\cdot(1+\log 4)\cdot A_1\cdot A_2\cdot A_3\cdot (1+\log2m)\\
   &<2\cdot 10^{16}\cdot (1+\log2m).
\end{align*}
The above inequality implies that $m<1.6\cdot 10^{17}$. We now need to reduce the upper bound obtained for $n$. For this, we use Lemma \ref{Lem1}.

Let $z=\log\left(\dfrac{C}{-D'}\right)+n\log\alpha-m\log\beta$. Observe that $z\neq0$ as $\Lambda\neq0$ for $n,m\geq2$. If $z>0$, then from (\ref{eq 10}), we have
$$0<z\le e^z-1<\dfrac{32}{\beta^m}.$$
If $z<0$, then $|e^z-1|<\dfrac{32}{\beta^m}<\dfrac12$ for $m\ge 2$, so
$|z|<2|e^z-1|<\dfrac{64}{\beta^m}$. In both cases
\begin{align*}
\left|\log\left(\frac{C}{-D'}\right)+n\log\alpha-m\log\beta\right|<\frac{64}{\beta^m}.
\end{align*}
After multiplying both sides of the above inequality by $\dfrac{1}{\log\beta}$, we obtain
\begin{align*}
\left|\frac{\log\left(\dfrac{C}{-D'}\right)}{\log\beta}+n\frac{\log\alpha}{\log\beta}-m\right|<\frac{64}{\log\beta\cdot\beta^m}<\frac{13}{\beta^m}.
\end{align*}
Now, let us apply Lemma \ref{Lem1} with
\begin{align*}
u=n,\quad \gamma=\frac{\log\alpha}{\log \beta},\quad v=m,\quad \mu=\frac{\log\left(\dfrac{C}{-D'}\right)}{\log\beta},\quad A=13,\quad B=\beta,\quad w=m.
\end{align*}
Since $m<1.6\cdot 10^{17}$ and $u=n<2m<3.2\cdot 10^{17}$, we set $M=3.2\cdot 10^{17}$. Choosing  
$$\frac{p}{q}=\frac{p_{37}}{q_{37}}=\frac{2487021452540632304}{3261842661237276437},$$
we get $q_{37}>6M$ and
$$\epsilon=||\mu q||-M||\gamma q||=0.3987\ldots>0.$$
Therefore, by Lemma \ref{Lem1}, we have
\begin{align*}
    m<\frac{\log(Aq/\epsilon)}{\log B}<9.
\end{align*}
Therefore, $m\leq 8$. Now, for $2\leq m\leq 8$, we have $\xi'_{1,n}\neq \eta_{2,m}$.\\
This completes the proof.
\end{proof}

\begin{proof}[Proof of Corollary \ref{Theorem 1.9}]
If possible assume that there exits positive integers $a\neq b$ such that $\{4,108,165,a,b\}$ is a star Diophantine $D(1)$ quintuple. Then
the sets $\{4,108,165,a\}$ and $\{4,108,165,b\}$ are a star Diophantine $D(1)$ quadruple, but by Theorem \ref{Theorem 1.8}, we have $a=b=360$, which is a contradiction.
\end{proof}

\section{Data Availability} 	
The authors confirm that their manuscript has no associated data.
\section{Competing Interests}
The authors confirm that they have no competing interests. No AI has been used for any part of this manuscript.


\begin{thebibliography}{99}
\bibitem{BCM}
N. C. Bonciocat, M. Cipu and M. Mignotte,
\textit{There is no Diophantine $D(-1)$-quadruple},
J. London Math. Soc. \textbf{105} (2022), 63--99.

\bibitem{D4}
M. Bliznac Trebje\v{s}anin and A. Filipin,
\textit{Nonexistence of $D(4)$-quintuples},
J. Number Theory \textbf{194} (2019), 170--217.

\bibitem{r4}
M. Bliznac Trebje\v{s}anin,
\textit{On Diophantine triples containing a triangular number},
Math. Slovaca (to appear) (2026).

\bibitem{r5}
M. Bliznac Trebje\v{s}anin,
\textit{On Diophantine pairs and triples of triangular numbers}, \href{https://doi.org/10.48550/arXiv.2603.29565}{arXiv.2603.29565} (2026).

\bibitem{r3}
S. Bagchi and C. Zhou-Zheng,
\textit{Diophantine $m$-tuples of triangular numbers}, \href{https://doi.org/10.48550/arXiv.2608.27697}{arXiv.2608.27697} (2026).

\bibitem{1st}
M. N. Deshpande,
\textit{One property of triangular numbers},
Portugaliae Math. \textbf{55} (1998), 381--383.

\bibitem{book}
L. E. Dickson,
\textit{History of the Theory of Numbers, Vol.~II},
Chelsea, New York, 1966, 514--519.

\bibitem{conjecture}
A. Dujella,
\textit{Generalization of a problem of Diophantus},
Acta Arith. \textbf{65} (1993), 15--27.

\bibitem{homepage}
A. Dujella,
Diophantine $m$-tuples,
\href{https://web.math.pmf.unizg.hr/~duje/ref.html}{web.math.pmf.unizg.hr/\string~duje/dtuples.html}.

\bibitem{r7}
A. Dujella and A. Peth{\H{o}},
A generalization of a theorem of Baker and Davenport,
\textit{Quart. J. Math. Oxford Ser.} (2) \textbf{49} (1998), no. 195, 291--306.

\bibitem{tri5}
A. Filipin, Y. Fujita and A. Jurasi\'c,
On the non-extendibility of triangular $D(n)$-pairs $\{n,2n\}$ with $n=2k^2\pm k$,
Acta Math. Hungar., to appear.

\bibitem{tri4}
A. Filipin, A. Jurasi\'c and A. Togb\'e,
A note on the extension of some triangular Diophantine sets,
\textit{An. \c{S}t. Univ. Ovidius Constan\c{t}a} \textbf{34}(2) (2026), 29--43.

\bibitem{tri1}
A. Filipin and L. Szalay,
Triangular Diophantine tuples from $\{1,2\}$,
\textit{Rad Hrvat. Akad. Znan. Umjet. Mat. Znan.} \textbf{27} (2023), 55--70.

\bibitem{tri2}
A. Filipin, A. Jurasi\'c and L. Szalay,
Extensions of the triangular $D(3)$-pair $\{3,6\}$,
\textit{Math. Slovaca} \textbf{75} (2025), 775--790.

\bibitem{Fib1}
Y. Fujita and F. Luca,
There are no Diophantine quadruples of Fibonacci numbers,
\textit{Acta Arith.} \textbf{185}(1) (2018), 19--38.

\bibitem{Hamtat}
A. Hamtat,
\textit{Diophantine triples in triangular numbers},
preprint (2025).

\bibitem{r1}
B. He, A. Togb\'e and V. Ziegler,
\textit{There is no Diophantine quintuple},
Trans. Amer. Math. Soc. \textbf{371} (2019), 6665--6709.

\bibitem{Fib2}
B. He, F. Luca and A. Togb\'e,
Diophantine triples of Fibonacci numbers,
\textit{Acta Arith.} \textbf{175} (2016), 57--70.

\bibitem{tri3}
M. Juki\'c Bokun and I. Soldo,
Triangular $D(-1)$-tuples,
Bull. Math. Soc. Sci. Math. Roumanie (2024), to appear.

\bibitem{r13}
E. M. Matveev,
An explicit lower bound for a homogeneous rational linear form in logarithms of algebraic numbers. II,
\textit{Izv. Ross. Akad. Nauk Ser. Mat.} \textbf{64} (2000), no. 6, 125--180.

\bibitem{r2}
Diophantus of Alexandria,
\textit{Arithmetics and the Book of Polygonal Numbers},
I. G. Bashmakova (ed.), Nauka, Moscow, 1974, 103--104, 232.

\bibitem{pell1}
S. E. Rihane, M. O. Hernane and A. Togb\'e,
On Diophantine triples of Pell numbers,
\textit{Colloq. Math.} \textbf{156} (2019), 273--285.

\bibitem{pell2}
S. E. Rihane, F. Luca and A. Togb\'e,
There are no Diophantine quadruples of Pell numbers,
\textit{Int. J. Number Theory} \textbf{18} (2022), 27--45.

\bibitem{r6}
P. Radi\'c,
\textit{On some $D(u^2)$-triples and quadruples of triangular numbers},
\href{https://doi.org/10.48550/arXiv.2609.28018}{arXiv.2609.28018} (2026).
\end{thebibliography}
\end{document}